\documentclass[11pt]{amsart}
\usepackage{amsfonts,amssymb,amsmath,amsthm,faktor}
\usepackage{enumitem,tikz}
\usepackage[colorlinks,citecolor=blue,urlcolor=black,linkcolor=black]{hyperref}

\newtheorem{thm}{Theorem}[section]
\newtheorem*{mainthm}{Main Theorem}
\newtheorem{lemma}[thm]{Lemma}
\newtheorem{cor}[thm]{Corollary}
\newtheorem{claim}{Claim}[thm]
\newtheorem{prop}[thm]{Proposition}
\newtheorem{fact}[thm]{Fact}

\theoremstyle{definition}
\newtheorem{defn}[thm]{Definition}
\newtheorem{notation}[thm]{Notation}

\theoremstyle{remark}
\newtheorem{remark}[thm]{Remark}
\newtheorem{conv}[thm]{Convention}

\renewcommand{\mid}{\mathrel{|}\allowbreak}
\renewcommand{\restriction}{\mathbin\upharpoonright}

\DeclareMathOperator{\crit}{crit}

\DeclareMathOperator{\acc}{acc}

\DeclareMathOperator{\ad}{AD}

\DeclareMathOperator{\nacc}{nacc}

\DeclareMathOperator{\dom}{dom}
\DeclareMathOperator{\im}{Im}

\DeclareMathOperator{\h}{ht}

\newcommand\s{\subseteq}
\newcommand\supp{{\circledcirc}}

\def\br{\blacktriangleright}

\def\pII{\textrm{\bf II}}
\newcounter{oldenumi}

\title{A new model for all $C$-sequences are trivial}

\author{Assaf Rinot}
\address{Department of Mathematics, Bar-Ilan University, Ramat-Gan 5290002, Israel.}
\urladdr{http://www.assafrinot.com}

\author{Zhixing You}
\address{Department of Mathematics, Bar-Ilan University, Ramat-Gan 5290002, Israel}
\email{zhixingy121@gmail.com}

\author{Jiachen Yuan}
\address{UK}
\email{yuan.jachen@gmail.com}


\begin{document}
\begin{abstract} We construct a model in which all $C$-sequences are trivial,
yet there exists a $\kappa$-Souslin tree with full vanishing levels.
This answers a question from \cite{paper58}, and provides an optimal combination of compactness and incompactness.
It is obtained by incorporating a so-called \emph{mutually exclusive ascent path} to Kunen's original forcing construction.
\end{abstract}
\date{Preprint as of \today. For updates, visit \textsf{http://p.assafrinot.com/70}.}
\maketitle

\section{introduction}
Motivated by a characterization of weak compactness in terms of $C$-sequences
due to Todor\v{c}evi\'c \cite[Theorem 1.8]{TodActa},
Lambie-Hanson and Rinot \cite[Definition 1.6]{paper35}
introduced a new cardinal characteristic, \emph{the $C$-sequence number},
to measure the compactness of a regular uncountable cardinal $\kappa$, where the best case $\chi(\kappa)=0$ amounts to saying that $\kappa$ is weakly compact.
This notion is quite useful, for example, if the $C$-sequence number of a strongly inaccessible $\kappa$ is bigger than $1$ (i.e., $\chi(\kappa)>1$),
then the $\kappa$-chain condition is not infinitely productive (see \cite[Lemma~3.3]{paper34} and \cite[Lemma~5.8]{paper35}),
and there is a $\kappa$-Aronszajn tree with no ascent path of width less than $\chi(\kappa)$ (see \cite[Lemma~7.9]{paper71}).\footnote{Undefined terminology can be found in Section~\ref{sp}.}

A simple way to argue that $\chi(\kappa)=1$ holds in a given model is to
prove that the weak compactness of $\kappa$ can be resurrected via some {$\kappa$-cc} notion of forcing.
As a sole example, Kunen's model \cite[\S3]{Kunen} satisfies this requirement.

In \cite{paper48}, Rinot and Shalev put forward the importance of the vanishing levels $V(T)$ of a $\kappa$-Souslin tree $T$,
deriving from it instances of the guessing principle $\clubsuit_{\ad}$
and using it to solve problems in set-theoretic topology.
They proved \cite[Theorem~2.30]{paper48} that for $\kappa$ weakly compact, $\clubsuit_{\ad}$ fails over every club in $\kappa$.
In addition, if $\clubsuit_{\ad}$ holds over a club in $\kappa$, then $\kappa$ is immediately seen to be non-subtle. This raises the question of what large cardinal notions are compatible with $\clubsuit_{\ad}$ holding over a club.
The best known result in this vein is \cite[Theorem~E]{paper58} asserting that assuming the consistency of a weakly compact cardinal,
it is consistent that for some strongly inaccessible cardinal $\kappa$ satisfying $\chi(\kappa) = \omega$,
$\clubsuit_{\ad}$ holds over a club in $\kappa$, furthermore,
there is a $\kappa$-Souslin tree $T$ such that $V(T) = \acc(\kappa)$.
Whether this result may be improved to $\chi(\kappa)<\omega$ remained an evasive open problem. Here, this remaining case is resolved.

\begin{mainthm}
Suppose that $\kappa$ is a weakly compact cardinal. Then there exists a forcing extension in which $\kappa$ is strongly inaccessible, $\chi(\kappa)=1$
and there exists a $\kappa$-Souslin tree $T$ such that $V(T)=\acc(\kappa)$.
\end{mainthm}

A natural attempt to prove the preceding is to follow Kunen's approach,
this time adding a generic $\kappa$-Souslin tree $T$ such that $V(T)$ covers a club in $\kappa$,
hopefully arguing that the weak compactness of $\kappa$ can be resurrected by a further $\kappa$-cc forcing.
However, in view of the implicit requirement to kill subtle-ness, this demands adding at least $\kappa^+$-many branches through $T$,
which has a negative effect on the chain condition of the further forcing.
Instead, our approach is to produce a model admitting a family of generic elementary embeddings living in further $\kappa$-cc forcing extensions, as follows.

We will start with $\kappa$ a weakly compact cardinal,
carry out some preparatory forcing below $\kappa$, and then force to add a uniformly homogeneous $\kappa$-Souslin subtree $T$ of ${}^{<\kappa}\kappa$
such that $V(T)$ is a club in $\kappa$ and such that $T$ admits an $\mathcal F$-ascent path for an educated choice of a filter $\mathcal F$ over $\omega$.
Our plan is to argue that $\chi(\kappa)=1$ holds in the final model by showing that for every $C$-sequence $\vec C$ over $\kappa$,
a nontrivial elementary embedding $j:M\rightarrow N$ between two $\kappa$-models with $\crit(j)=\kappa$ and $\vec C\in M$
exists in some further $\kappa$-cc forcing extension.
In the inevitable case that our tree $T$ belongs to $M$, it would admit a branch $b:\kappa\rightarrow\kappa$,
as witnessed by any element of the $\kappa^{\text{th}}$-level of $j(T)$.
Meanwhile, since $V(T)$ covers a club, $\kappa\in j(V(T))=V(j(T))$, and hence the tree $j(T)$ would have a vanishing $\kappa$-branch $b':\kappa\rightarrow\kappa$.
As $b$ is non-vanishing and $b'$ is vanishing, the fact that $T$ is uniformly homogeneous
implies that $b$ and $b'$ must disagree cofinally often.
By iterating such an argument, we infer that the tree must have at least $\omega$-many cofinal branches that disagree with each other cofinally often.
Thus, we shall need a notion of forcing that introduces such cofinal branches in a $\kappa$-cc fashion.
This is exactly where the definition of a \emph{mutually exclusive $\mathcal F$-ascent path} arises.
To connect on a previous remark, instead of adding $\kappa^+$-many branches, in our approach only countably many (mutually exclusive ones) are added.

\subsection{Organization of this paper}
In Section~\ref{sp}, we provide some necessary background on trees, ascent paths and $C$-sequences. We motivate the new notion of a \emph{mutually exclusive $\mathcal{F}$-ascent path}
by showing that variations of Kunen's forcing that do not make use of this kind of ascent path are not well behaved.

In Section~\ref{sec3}, we present our main notions of forcing. Large cardinals play no role in this section, and the results are applicable as low as at $\aleph_2$.

In Section~\ref{sect7}, we prove the main theorem by iterating the posets of Section~\ref{sec3} below and at a weakly compact cardinal.

\section{Preliminaries}\label{sp}
\subsection{Trees and vanishing levels}
For simplicity, instead of working with abstract trees, we opt to work with the following more concrete implementation.
A \emph{streamlined tree} is a subset $T\s{}^{<\kappa}H_\kappa$ for some cardinal $\kappa$
that is downward-closed, i.e., for every $t\in T$, $\{ t\restriction \alpha\mid \alpha<\kappa\}\s T$.
The height of a node $x\in T$ is $\dom(x)$. The height of $T$, denoted $\h(T)$, is the least ordinal $\beta$ such that $T_\beta:=\{x\in T\mid \dom(x)=\beta\}$ is empty.
Note that for every ordinal $\beta$, $T\restriction\beta:=\{x\in T\mid \dom(x)<\beta\}$ is as well a streamlined tree.
A streamlined tree $T$ is said to be \emph{normal} iff for every $x\in T$ and every $\alpha<\h(T)$,
there exists some $y\in T_\alpha$ with either $x\s y$ or $y\s x$.
A \emph{streamlined $\kappa$-tree} {\cite[Definition~2.3]{paper23}} is a downward-closed subset $T$ of ${}^{<\kappa}H_\kappa$
satisfying that $0<|T_\alpha|<\kappa$ for every $\alpha<\kappa$.\footnote{In this case, the poset $(T,{\subsetneq})$
is a set-theoretic $\kappa$-tree in the usual abstract sense.}
A \emph{streamlined $\kappa$-Souslin tree} is a streamlined $\kappa$-tree having no $\kappa$-branches and no $\kappa$-sized antichains with respect to the ordering $\subsetneq$.
\begin{conv}
Hereafter, by a tree, we mean a streamlined tree.
\end{conv}

\begin{defn} For all $s,t\in{}^{<\kappa}H_\kappa$, let
$$\Delta(s,t):=\min(\{\dom(s),\dom(t)\}\cup\{\delta\in\dom(s)\cap\dom(t)\mid s(\delta)\neq t(\delta)\}).$$
In addition, we define $(s*t):\dom(t)\rightarrow H_\kappa$ via:
$$(s*t)(\varepsilon):=\begin{cases}s(\varepsilon),&\text{if }\varepsilon\in\dom(s);\\
t(\varepsilon),&\text{otherwise.}\end{cases}$$
\end{defn}
\begin{defn} A $\kappa$-tree $T$ is
\emph{uniformly homogeneous} iff for all $s, t\in T$, $s*t$ is in $T$.
\end{defn}

\begin{defn} Suppose that $T$ is a tree.
An \emph{$\alpha$-branch} is a subset $B\s T$ linearly ordered by $\subsetneq$ and satisfying that $\{ \dom(x)\mid x\in B\}=\alpha$.
For $\alpha\in\acc(\kappa)$, an $\alpha$-branch $B$ is \emph{vanishing} iff $\bigcup B\notin T$.
\end{defn}

\begin{defn}[The levels of vanishing branches, {\cite[Definition~2.18]{paper48}}]\label{defvanish}
For a $\kappa$-tree $T$, $ V(T)$ stands for the set of all $\alpha\in\acc(\kappa)$
such that for every node $x\in T$ of height less than $\alpha$ there exists a vanishing $\alpha$-branch containing $x$.
\end{defn}

In the context of uniformly homogeneous trees, the preceding admits a simpler characterization
(see \cite[Proposition~2.6]{paper58}):

\begin{fact} For a uniformly homogeneous $\kappa$-tree $T$, $V(T)$ coincides with the set of all $\alpha\in\acc(\kappa)$
for which there exists a vanishing $\alpha$-branch.
\end{fact}

\subsection{Ascent paths}
The notions of forcing in Section~\ref{sec3} below will make use
of the upcoming Definitions \ref{def26}, \ref{def211} and \ref{def212}. To motivate them, we shall inspect here two earlier attempts to define these notions of forcing,
demonstrating the problem with these attempts.

Let $\theta<\kappa$ be a pair of infinite regular cardinals.
We start by recalling the vanilla definition of an ascent path.

\begin{defn}\label{def26}
Suppose that $T$ is a tree, and $\theta$ is some cardinal. For all $f,g \in {}^{\theta}T$,
denote $$\supp(f,g):=\{ \tau<\theta \mid f(\tau) \s g(\tau)\text{ or }f(\tau) \supseteq g(\tau)\}.$$
\end{defn}

\begin{defn}[Laver]
Suppose that $T$ is a tree of some height $\gamma$.
A sequence $\vec f=\langle f_\alpha \mid \alpha < \gamma \rangle$ is a \emph{$\theta$-ascent path} through $T$ iff the following two hold:
\begin{itemize}
\item for every $\alpha < \gamma$, $f_\alpha:\theta \rightarrow T_\alpha$ is a function;
\item for all $\alpha < \beta < \gamma$, $\supp(f_{\alpha},f_{\beta})$ is co-bounded in $\theta$.
\end{itemize}
\end{defn}

A natural notion of forcing for adding a $\kappa$-Souslin tree with a $\theta$-ascent path reads as follows.
\begin{defn}\label{defPkt}
$\mathbb S^\kappa_{\theta}$
is defined to be the notion of forcing consisting of all pairs
$\langle T,\vec{f} \rangle$ for which the following two hold:
\begin{enumerate}[label=(\arabic*)]
\item $T \subseteq {^{<\kappa}}\kappa$ is a normal uniformly homogeneous tree of a successor height
all of whose levels have size less than $\kappa$;
\item $\vec{f}$ is a $\theta$-ascent path through $T$.
\end{enumerate}
The order on $\mathbb S^\kappa_{\theta}$ is defined by taking end-extension on both coordinates.
\end{defn}

Work in $V[G]$ for $G$ an $\mathbb S^\kappa_{\theta}$-generic filter over $V$.
It can be verified that
$T(G):=\bigcup\{T\mid \langle T,\vec{f}\rangle \in G\}$ is a uniformly homogeneous $\kappa$-Souslin tree,
and $\vec{f^G}:=\bigcup\{\vec{f}\mid \langle T,\vec{f}\rangle \in G\}$ is a $\theta$-ascent path through $T(G)$.
Next, consider the following further forcing.
\begin{defn} $\mathbb A^\kappa_{\theta}$ has underlying set $\kappa$,
and ordering
$$\beta\leq_{\mathbb A^\kappa_{\theta}}\alpha \text{ iff }(\alpha\le\beta\ \&\ \supp(\vec{f^G}(\alpha),\vec{f^G}(\beta))=\theta).$$
\end{defn}

\begin{prop} $\mathbb A^\kappa_{\theta}$ does not satisfy the $\kappa$-chain condition.
\end{prop}
\begin{proof}
Let us say that $\beta<\kappa$ is \emph{bad} iff $\beta=\alpha+1$ is a successor ordinal
and there are $\tau<\tau'<\theta$ such that $\Delta(\vec{f^G}(\beta)(\tau),\vec{f^G}(\beta)(\tau'))=\alpha$.
\begin{claim} There are cofinally many bad $\beta<\kappa$.
\end{claim}
\begin{proof} Back in $V$, given a condition $\langle T,\vec{f}\rangle$ in $\mathbb S^\kappa_\theta$,
fix $\alpha$ such that $\h(T)=\alpha+1$.
Let $T':=T\cup\{ t{}^\smallfrown\langle \iota\rangle\mid t\in T_\alpha, \iota<\theta\}$,
and define $\vec f':\alpha+2\rightarrow{}^\theta{T'}$ via
$$\vec f'(\beta)(\tau):=\begin{cases}
\vec f(\beta)(\tau),&\text{if }\beta\le\alpha;\\
\vec f(\alpha)(\tau){}^\smallfrown\langle\tau\rangle,&\text{if }\beta=\alpha+1\ \&\ \tau\ge2;\\
\vec f(\alpha)(0){}^\smallfrown\langle\tau\rangle,&\text{otherwise}.
\end{cases}$$
Then $\langle T',\vec f'\rangle$ extends $\langle T,\vec f\rangle$ and it forces that $\dom(\vec f')$ is bad.
\end{proof}

As the collection of all bad $\beta<\kappa$ is an antichain in $\mathbb A^\kappa_\theta$,
the latter fails to satisfy that $\kappa$-cc.
\end{proof}

To mitigate the problem arising from the preceding proposition, and in view of the goal of securing a good chain condition, we introduce the next two definitions.
\begin{defn}\label{def211}
Two elements $s,t\in {}^{<\kappa}H_\kappa$ are \emph{mutually exclusive} iff $s(\varepsilon)\neq t(\varepsilon)$ for every $\varepsilon \in \dom(s)\cap \dom(t)$.
\end{defn}

\begin{defn}\label{def212}
Suppose that $T$ is a tree of some height $\gamma$,
and $\mathcal F$ is a filter over $\theta$.
A sequence $\vec f=\langle f_\alpha \mid \alpha < \gamma \rangle$ is a \emph{mutually exclusive $\mathcal{F}$-ascent path} through $T$ iff the following three hold:
\begin{itemize}
\item for every $\alpha < \gamma$, $f_\alpha:\theta \rightarrow T_\alpha$ is a function;
\item for all $\alpha < \beta < \gamma$, $\supp(f_{\alpha},f_{\beta})\in \mathcal{F}$;
\item for every nonzero $\alpha<\gamma$, $\langle f_\alpha(\tau)\mid \tau<\theta\rangle$ consists of mutually exclusive nodes.
\end{itemize}
\end{defn}

Fix some uniform filter $\mathcal F$ over $\theta$, and revise Definition~\ref{defPkt}, as follows.
\begin{defn}
$\mathbb S^\kappa_{\mathcal F}$
is defined to be the notion of forcing consisting of all pairs
$\langle T,\vec{f} \rangle$ for which the following two hold:
\begin{enumerate}[label=(\arabic*)]
\item $T \subseteq {^{<\kappa}}\kappa$ is a normal uniformly homogeneous tree of a successor height
all of whose levels have size less than $\kappa$;
\item $\vec{f}$ is a mutually exclusive $\mathcal F$-ascent path through $T$.
\end{enumerate}
The order on $\mathbb S^\kappa_{\mathcal F}$ is defined by taking end-extension on both coordinates.
\end{defn}

Work in $V[G]$ for $G$ an $\mathbb S^\kappa_{\mathcal F}$-generic filter over $V$.
Define $\mathbb A^\kappa_{\mathcal F}$ to have underlying set $\kappa$,
and ordering
$$\beta\leq_{\mathbb A^\kappa_{\mathcal F}}\alpha \text{ iff }(\alpha\le\beta\ \&\ \supp(\vec{f^G}(\alpha),\vec{f^G}(\beta))\in\mathcal F).$$

This time --- for an educated choice of a filter $\mathcal F$ ---
the corresponding notion of forcing $\mathbb A^\kappa_{\mathcal F}$ does satisfy the $\kappa$-chain condition.
So we are almost done: only a small step away from making sure that $T(G)$ is a $\kappa$-Souslin tree.\footnote{During a talk given by the second author at the Chinese Academy of Sciences, Yinhe Peng demonstrated that $T(G)$ has an antichain of size $\kappa$.}
However, this is a benign problem, since the subtree generated by the ascent path is always a $\kappa$-Souslin tree.

In Section~\ref{sec3}, we shall present a pair $(\mathbb S^\kappa_{\mathbf X},\mathbb A^\kappa_{\mathbf X})$
of notions of forcing that overcome the problems detected in the two earlier attempts $(\mathbb S^\kappa_{\theta},\mathbb A^\kappa_{\theta})$
and $(\mathbb S^\kappa_{\mathcal F}, \mathbb A^\kappa_{\mathcal F})$ considered here.
Specifically, $\mathbb S^\kappa_{\mathbf X}$ will be a ${<}\kappa$-strategically closed forcing that adds a uniformly homogeneous $\kappa$-Souslin tree with a
mutually exclusive $\mathcal{F}$-ascent path
and whose set of vanishing levels is a club in $\kappa$. As well, $\mathbb S^\kappa_{\mathbf X}*\dot{\mathbb A}^\kappa_{\mathbf X}$
will be a ${<}\kappa$-strategically closed forcing, and ${\mathbb A}^\kappa_{\mathbf X}$ will satisfy the $\kappa$-chain condition.

\subsection{$C$-sequences}
A \emph{$C$-sequence} over a regular uncountable cardinal $\kappa$
is a sequence $\vec C=\langle C_\beta\mid\beta<\kappa\rangle$
such that for every $\beta<\kappa$, $C_\beta$ is a closed subset of $\beta$ with $\sup(C_\beta)=\sup(\beta)$.

\begin{defn}[The $C$-sequence number of $\kappa$, \cite{paper35}]\label{defcnm}
If $\kappa$ is weakly compact, then let $\chi(\kappa):=0$. Otherwise, let
$\chi(\kappa)$ denote the least cardinal $\chi\le\kappa$
such that, for every $C$-sequence $\langle C_\beta\mid\beta<\kappa\rangle$,
there exist $\Delta\in[\kappa]^\kappa$ and $b:\kappa\rightarrow[\kappa]^{\chi}$
with $\Delta\cap\alpha\s\bigcup_{\beta\in b(\alpha)}C_\beta$
for every $\alpha<\kappa$.
\end{defn}
\begin{remark}
The special case of $\chi(\kappa)\le1$ is better known as ``all $C$-sequences over $\kappa$ are trivial''.
\end{remark}
To motivate one ingredient of the proof of the main result in Section~\ref{sect7}, we present here an abstract approach for producing models satisfying $\chi(\kappa)=1$.
It may be phrased as a \emph{suspended form of the tree property}, as follows.
\begin{prop}\label{lemma69} Suppose that $\kappa$ is a strongly inaccessible cardinal,
and for every $\kappa$-tree $T$,
there is a $\kappa$-cc forcing extension in which $T$ has a $\kappa$-branch.
Then $\chi(\kappa)\le 1$.
\end{prop}
\begin{proof} Let $\vec C=\langle C_\beta\mid\beta<\kappa\rangle$ be a $C$-sequence over $\kappa$.
Let $T$ be the corresponding tree $T(\rho_0^{\vec C})$ from the theory of walks on ordinals.
Since $\kappa$ is strongly inaccessible, $T$ is narrow, i.e., it is a $\kappa$-tree.
Let $V'$ be some $\kappa$-cc forcing extension of the universe in which $T$ admits a $\kappa$-branch.
By the proof of the forward implication of \cite[Theorem~6.3.5]{MR2355670}, there must exists a club $C\s\kappa$ such that,
for every $\alpha<\kappa$, there is a $\beta\ge\alpha$ with $C\cap\alpha=C_\beta\cap\alpha$.
Since $V'$ is a $\kappa$-cc forcing extension of $V$, we may find a subclub $D\s C$ residing in $V$.
It follows that for every $\alpha<\kappa$, there is a $\beta\ge\kappa$ with $D\cap\alpha\s C_\beta$.
\end{proof}
\begin{remark} An analogous statement for $\chi(\kappa)=0$ follows from \cite[Lemma~2.2]{MR3072773}.
\end{remark}

\section{An ascent path variant of Kunen's construction}\label{sec3}

Throughout this section, $\theta,\kappa$ stand for infinite regular cardinals and we assume that $\lambda^\theta<\kappa$ for all $\lambda<\kappa$.
In particular, $\theta^+<\kappa$.
We also fix a sequence $\mathbf X=\langle X_\xi\mid\xi<\zeta\rangle$ satisfying the following requirements:
\begin{itemize}
\item $\langle X_\xi\mid \xi<\zeta\rangle$ is a $\s$-decreasing sequence of nonempty subsets of $\theta$;
\item $|\theta\setminus X_0|=\theta$;
\item $\bigcap_{\xi<\zeta}X_\xi=\emptyset$.
\end{itemize}

We denote by $\mathcal F$ the filter generated by the elements of $\mathbf X$,
i.e., $\mathcal F=\{ Y\s\theta\mid\exists\xi<\zeta\,(X_\xi\s Y)\}$.
Note that $\mathcal F$ is proper and nonprincipal.
Also, let $\mathcal{I}$ denote the dual ideal of $\mathcal{F}$.

\begin{defn}
Define an equivalence relation $=^*$ on ${}^{<\kappa}H_\kappa$ by letting
$s=^*t$ iff $\dom(s)=\dom(t)$ and there exists an $\alpha<\dom(s)$ such that $s(\beta) = t(\beta)$ whenever $\alpha \leq \beta<\dom(s)$.
\end{defn}

We are now ready to present the notions of forcing we will be using.

\begin{defn}\label{def31}
$\mathbb S^\kappa_{\mathbf X}$
is defined to be the notion of forcing consisting of all pairs
$\langle T,\vec{f} \rangle$ that satisfy the following list of requirements:
\begin{enumerate}[label=(\arabic*)]
\item $T \subseteq {^{<\kappa}}\kappa$ is a normal uniformly homogeneous tree of a successor height, say $\eta+1$,
all of whose levels have size less than $\kappa$;
\item\label{ascentpathclause} $\vec{f}$ is a mutually exclusive $\mathcal{F}$-ascent path through $T$;
\item\label{vanishingclause} $V(T)$ is a closed set, and if $\eta$ is a limit ordinal, then $\eta \in V(T)$;
\item\label{keyclause} for every nonzero $\alpha\le\eta$, for every $t \in T_\alpha$, the set
$\{\tau<\theta \mid t \text{ and }\vec f(\alpha)(\tau)\text{ are mutually exclusive}\}$ is in $\mathcal{F}$.\footnote{This requirement ensures that the generic tree coincides with the one generated by its ascent path. It overcomes the issue with the forcing $\mathbb S^\kappa_{\mathcal F}$ discussed in the previous section.}
\end{enumerate}
The order on $\mathbb S^\kappa_{\mathbf X}$ is defined by taking end-extension on both coordinates.
\end{defn}

\begin{notation} For $\langle T,\vec{f} \rangle \in \mathbb S^\kappa_{\mathbf X}$:
\begin{itemize}
\item $\eta_T$ denotes the unique ordinal to satisfy $\h(T)=\eta_T+1$;
\item $\nu_T$ denotes the cardinality of the quotient $\faktor{T_{\eta_T}}{=^*}$.
\end{itemize}
\end{notation}
\begin{remark} If $\eta_T=\alpha+1$ is a successor ordinal, then $\nu_T=|\{t(\alpha)\mid t\in T_{\eta_T}\}|$.
\end{remark}

\begin{lemma}[One-step $\beta$-extension]\label{onestepextension2} Let $\langle T,\vec{f} \rangle \in \mathbb S^\kappa_{\mathbf X}$,
$\beta\le\eta_T$ and $\nu<\kappa$. Then there is a $\langle T',\vec{f'}\rangle\le \langle T,\vec{f}\rangle$ with
$\eta_{T'}=\eta_T+1$ and $\nu_{T'}\ge\nu$ such that:
\begin{itemize}
\item $\supp(\vec{f}(\beta),\vec{f}(\eta_T))\s \supp(\vec{f}(\eta_T),\vec{f'}(\eta_{T'}))$, and
\item $\supp(\vec{f}(\beta),\vec{f'}(\eta_{T'}))=\theta$.
\end{itemize}
\end{lemma}
\begin{proof} Set $T':=T\cup\{ t{}^\smallfrown\langle\iota\rangle\mid t\in T_{\eta_T},~\iota<|\nu\cup\theta|\}$,
and then define $\vec{f'}:\eta_{T'}+1\rightarrow{}^\theta({}^{<\kappa}\kappa)$ by letting for all $\alpha\le\eta_{T'}$ and $\tau<\theta$:
$$\vec{f'}(\alpha)(\tau):=\begin{cases}
\vec{f}(\alpha)(\tau),&\text{if }\alpha\le\eta_{T};\\
(\vec{f}(\beta)(\tau)*\vec{f}(\eta_T)(\tau)){}^\smallfrown\langle \tau\rangle,&\text{otherwise}.
\end{cases}$$

It is clear that $\langle T',\vec{f'}\rangle$ is an element of $\mathbb S^\kappa_{\mathbf X}$ as sought.
\end{proof}

The next lemma identifies a feature of decreasing sequences of conditions sufficient to ensure the existence of a lower bound.

\begin{lemma}\label{minimalextension2} Suppose $\Gamma$ is a cofinal subset of some
$\gamma\in\acc(\kappa)$, $\delta\in(\gamma,\kappa)$,
and we are given a sequence $\langle (T^\beta,\vec{f^\beta}, \vec{z}^\beta) \mid \beta\in\Gamma\rangle$ such that
\begin{itemize}
\item $\langle \langle T^\beta,\vec{f^\beta}\rangle\mid \beta\in\Gamma\rangle$ is a
decreasing sequence of conditions in $\mathbb S^\kappa_{\mathbf X}$;
\item for every $\beta\in\Gamma$, $\vec{z}^\beta:(\beta,\delta)\rightarrow T^\beta$, where $\beta<i<\delta$ implies:
\begin{itemize}
\item $\vec{z}^\beta(i)$ belongs to the top level of $T^\beta$;
\item $\vec{z}^\beta(i)\not=^*\vec{z}^\beta(i')$ for every $i'\in(\beta,\delta)\setminus\{i\}$;
\item $\vec{z}^\beta(i)\not=^*\vec{f}^\beta(\eta_{T^\beta})(0)$;
\item $\vec{z}^\beta(i)$ and $\vec{f}^\beta(\eta_{T^\beta})(\tau)$ are mutually exclusive whenever $0<\tau<\theta$.\footnote{This clause together with the previous one provide better control on how Clause~\ref{keyclause} of Definition~\ref{def31} is implemented with respect to the nodes enumerated by $\vec z^\beta$.}

\end{itemize}
\item for all $\alpha,\beta\in\Gamma$:
\begin{itemize}
\item $\supp(\vec{f^\alpha}(\eta_{T^\alpha}),\vec{f^\beta}(\eta_{T^\beta}))=\theta$;
\item $\vec z^\alpha(i) \s \vec z^\beta(i)$ provided $\alpha<\beta<i<\delta$.
\end{itemize}
\end{itemize}

Then there are a condition $\langle T^\gamma,\vec{f^\gamma} \rangle$ in $\mathbb S^\kappa_{\mathbf X}$ and $\vec{z^\gamma}:(\gamma,\delta)\rightarrow T^\gamma$ such that:
\begin{itemize}
\item $\eta_{T^\gamma}=\sup_{\beta\in\Gamma}\eta_{T^\beta}$;
\item $\vec{z^\gamma}(i)=\bigcup_{\beta\in\Gamma}\vec{z^\beta}(i)$ whenever $\gamma<i<\delta$;
\item for every $\beta\in\Gamma$, $\langle T^\gamma,\vec{f^\gamma}\rangle\le \langle T^\beta,\vec{f^\beta}\rangle$ and $\supp(\vec{f^\beta}(\eta_{T^\beta}),\vec{f^\gamma}(\eta_{T^\gamma}))=\theta$;
\item a function $y:\eta_{T^\gamma}\rightarrow\kappa$ is in $T^\gamma$ iff one of the following holds:
\begin{itemize}
\item there are $x\in\bigcup_{\beta\in\Gamma}T^\beta$ and $\tau<\theta$ such that $y=x*\vec{f^\gamma}(\eta_{T^\gamma})(\tau)$, or
\item there are $x\in\bigcup_{\beta\in\Gamma}T^\beta$ and $i\in(\gamma,\delta)$ such that $y=x*\vec z^\gamma(i)$.
\end{itemize}
\end{itemize}
\end{lemma}
\begin{proof}
Set $\eta:=\sup_{\beta\in\Gamma}\eta_{T^\beta}$ and $\vec{f}:=\bigcup_{\beta\in\Gamma}\vec{f}^{\beta}$.
Define $\vec{f^\gamma}:\eta+1\rightarrow{}^\theta({}^{<\kappa}\kappa)$ by letting for all $\alpha\le\eta$ and $\tau<\theta$:
$$\vec{f^\gamma}(\alpha)(\tau):=\begin{cases}
\vec{f}(\alpha)(\tau),&\text{if }\alpha<\eta;\\
\bigcup_{\beta\in\Gamma}\vec{f}(\eta_{T^\beta})(\tau),&\text{if }\alpha=\eta.
\end{cases}$$
Define $\vec{z^\gamma}:(\gamma,\delta)\rightarrow T^\gamma$ via
$$\vec{z^\gamma}(i):=\bigcup_{\beta\in\Gamma}\vec{z^\beta}(i).$$
Finally, consider $T:=\bigcup_{\beta\in\Gamma}T^\beta$, and then let
$$T^\gamma:=T \cup \bigcup \{ x*\vec{f^\gamma}(\eta)(\tau), x*\vec{z^\gamma}(i)\mid x \in T, \tau<\theta, \gamma<i<\delta \}.$$

To see that $\langle T^\gamma,\vec{f}^{\gamma}\rangle$ is as sought
note that
$$y:=\bigcup_{\beta\in\Gamma}\vec{z^\beta}(\gamma)$$
determines a vanishing $\eta$-branch of $T^\gamma$
and hence $\eta\in V^-(T)=V(T)$.
Also, for every $i \in (\gamma,\delta)$,
$\vec{z^\gamma}(i)$ and $\vec{f}^\gamma(\eta)(\tau)$ are mutually exclusive for every nonzero $\tau<\theta$.
So, for every $t\in (T^\gamma)_\eta$,
an intersection of two filter sets demonstrates that
$\{\tau<\theta \mid t \text{ and }\vec f^\gamma(\eta)(\tau)\text{ are mutually exclusive}\}$ is in $\mathcal{F}$.
\end{proof}

\begin{lemma}\label{cor282}
$\mathbb S^\kappa_{\mathbf X}$ is ${<}\kappa$-strategically closed.
\end{lemma}
\begin{proof} By Lemma~\ref{minimalextension2}. See the proof of the upcoming Lemma~\ref{Strclosure2} for a stronger result.
\end{proof}

Given an $\mathbb S^\kappa_{\mathbf X}$-generic filter $G$,
we let
$T(G):=\bigcup\{T\mid \langle T,\vec{f}\rangle \in G\}$, and
$\vec{f^G}:=\bigcup\{\vec{f}\mid \langle T,\vec{f}\rangle \in G\}$.

\begin{defn}\label{defA2}
Suppose $G$ is an $\mathbb S^\kappa_{\mathbf X}$-generic filter over $V$.
In $V[G]$, for every $\xi<\zeta$, define the forcing notion $\mathbb A^\kappa_{\mathbf X,\xi}$ associated with $G$ to have underlying set $\kappa$,
and ordering
$$\beta\leq_{\mathbb A^\kappa_{\mathbf X,\xi}}\alpha \text{ iff }(\alpha\le\beta\ \&\ X_\xi\s \supp(\vec{f^G}(\alpha),\vec{f^G}(\beta))).$$

Let $\mathbb A^\kappa_{\mathbf X}$ be the lottery sum of these $\mathbb A^\kappa_{\mathbf X,\xi}$'s over all $\xi<\zeta$.
\end{defn}

\begin{lemma}\label{Strclosure2}
$\mathbb S^\kappa_{\mathbf X}* \mathbb{\dot{A}}^\kappa_{\mathbf X}$ has a dense subset that is ${<}\kappa$-strategically closed.
\end{lemma}
\begin{proof} Let $\xi<\zeta$, and we shall prove that $\mathbb S^\kappa_{\mathbf X}* \mathbb{\dot{A}}^\kappa_{\mathbf X,\xi}$ has a dense subset that is ${<}\kappa$-strategically closed.

Let $D$ be the collection of all pairs $(p,\dot q)$ in $\mathbb S^\kappa_{\mathbf X}*\dot{\mathbb A}^\kappa_{\mathbf X,\xi}$
such that if $p=\langle T,\vec{f} \rangle$, then $\dot q=\check{\eta_T}$.
By Lemma~\ref{onestepextension2}, $D$ is dense in $\mathbb S^\kappa_{\mathbf X}*\dot{\mathbb A}^\kappa_{\mathbf X,\xi}$.
For simplicity, we shall identify $D$ with the collection of all triples $(T,\vec{f},\eta)$ such that $\langle T,\vec{f}\rangle\in \mathbb S^\kappa_{\mathbf X}$ and $\eta=\eta_T$,
where a triple $(T',\vec{f'} ,\eta')$ extends $( T,\vec{f} ,\eta)$ iff all of the following hold:
\begin{itemize}
\item $T'\cap{}^{\le\eta}\kappa=T$,
\item $\vec{f'}\restriction(\eta+1)=\vec f$, and
\item $X_\xi\s \supp(\vec f(\eta),\vec{f'}(\eta'))$.
\end{itemize}

To prove that $D$ is ${<}\kappa$-strategically closed,
let $\mu<\kappa$ be arbitrary,
and we shall describe a winning strategy for $\pII$ in the game $\Game_\mu(D)$.
Playing the game will yield a decreasing sequence of conditions in $D$, $\langle (T^{\beta},\vec{f}^{\beta},\eta_{\beta})\mid \beta<\mu \rangle$,
along with an auxiliary sequence $\langle \vec{z}^{\beta}\mid \beta<\mu\text{ nonzero even ordinal}\rangle$.

We start by letting $(T^0,\vec{f^0},\eta_0)$ be the maximal element of $D$, that is $T^0:=\{\emptyset\}$, $\vec{f^0}:1\rightarrow{}^\theta\{\emptyset\}$ and $\eta_0:=0$.
Next, suppose that $\gamma<\mu$ is a nonzero even ordinal and we have already obtained a decreasing sequence of conditions $\langle (T^{\beta},\vec{f}^{\beta},\eta_{\beta})\mid \beta<\gamma \rangle$ in $D$,
and a sequence $\langle \vec{z}^{\beta}\mid \beta<\gamma\text{ nonzero even ordinal}\rangle$
in such a way that all of the following hold:
\begin{enumerate}[label=(\roman*)]
\item\label{basic22} for all $\alpha<\beta<\gamma$, $X_\xi\s\supp(\vec f^{\alpha}(\eta_\alpha), \vec f^{\beta}(\eta_\beta))$;
\item\label{basic02} for every nonzero even $\beta<\gamma$, $\vec{z}^\beta:(\beta,\mu)\rightarrow T^\beta$, where $\beta<i<\mu$ implies:
\begin{itemize}
\item $\vec{z}^\beta(i)$ belongs to the top level of $T^\beta$;
\item $\vec{z}^\beta(i)\not=^*\vec{z}^\beta(i')$ for every $i'\in(\beta,\mu)\setminus\{i\}$;
\item $\vec{z}^\beta(i)\not=^*\vec{f}^\beta(\eta_\beta)(0)$;
\item $\vec{z}^\beta(i)$ and $\vec{f}^\beta(\eta_\beta)(\tau)$ are mutually exclusive for every nonzero $\tau<\theta$;
\end{itemize}
\item\label{basic12} for every pair $\alpha<\beta$ of even ordinals below $\gamma$:
\begin{itemize}
\item $\supp(\vec f^{\alpha}(\eta_\alpha), \vec f^{\beta}(\eta_\beta))=\theta$;
\item $\vec z^\alpha(i) \s \vec z^\beta(i)$ provided $0<\alpha<\beta<i<\mu$.
\end{itemize}
\end{enumerate}

Our description of the $\gamma^{\text{th}}$-move is divided into three cases, as follows.

$\br$ If $\gamma=2$, then using Lemma~\ref{onestepextension2},
fix a one-step $\eta_1$-extension $\langle T^2,\vec{f}^2\rangle$ of $\langle T^1,\vec f^1\rangle$
such that $\nu_{T^2}\ge\max\{\mu,\theta^+\}$.
Letting $\eta_2:=\eta_{T^2}$,
it is clear $(T^2,\vec f^2,\eta_2)$ belongs to $D$ and extends $(T^1,\vec f^1,\eta_1)$.
Now, as $\nu_{T^2}\ge\max\{\mu,\theta^+\}$,
we may fix a map $\vec{z^2}:(2,\mu)\rightarrow T^2$, where $2<i<\mu$ implies:
\begin{itemize}
\item $\vec{z^2}(i)$ belongs to the top level of $T^2$;
\item $\vec{z^2}(i)\not=^*\vec{z^2}(i')$ for every $i'\in(2,\mu)\setminus\{i\}$;
\item $\vec{z^2}(i)\not=^*\vec{f}^2(\eta_{2})(0)$;
\item $\vec{z^2}(i)$ and $\vec{f}^2(\eta_{2})(\tau)$ are mutually exclusive for every nonzero $\tau<\theta$.
\end{itemize}
Indeed, As $\eta_2=\eta_1+1$, this can be achieved by ensuring that $\langle \vec{z^2}(i)(\eta_1)\mid 2<i<\mu\rangle{}^\smallfrown\langle \vec{f^2}(\tau)(\eta_1)\mid \tau<\theta\rangle$
be an injective sequence, whereas $\langle \vec{z^2}(i)\restriction\eta_1\mid 2<i<\mu\rangle$
be a constant sequence whose sole element is $\vec{f^2}(\eta_2)(0)\restriction\eta_1$.

$\br$ If $\gamma$ is a successor ordinal, say, $\gamma=\alpha+2$,
then using Lemma~\ref{onestepextension2},
fix a one-step $\eta_\alpha$-extension $\langle T^\gamma,\vec{f}^\gamma\rangle$ of $\langle T^{\alpha+1},\vec f^{\alpha+1}\rangle$
such that $\nu_{T^\gamma}\ge\max\{\mu,\theta^+\}$.
Letting $\eta_\gamma:=\eta_{T^\gamma}$,
it is clear $(T^\gamma,\vec f^\gamma,\eta_\gamma)$ belongs to $D$,
and
$$X_\xi\s \supp(\vec f^{\alpha}(\eta_{\alpha}),\vec{f}^{\alpha+1}(\eta_{\alpha+1}))\s \supp(\vec f^{\alpha+1}(\eta_{\alpha+1}),\vec{f}^\gamma(\eta_\gamma)).$$

Next, we fix a map $\vec{z}^\gamma:(\gamma,\mu)\rightarrow T^\gamma$, where $\gamma<i<\mu$ implies:
\begin{itemize}
\item $\vec{z}^\gamma(i)$ belongs to the top level of $T^\gamma$;
\item $\vec{z}^\gamma(i)\not=^*\vec{z}^\gamma(i')$ for every $i'\in(\gamma,\mu)\setminus\{i\}$;
\item $\vec z^\alpha(i)\s \vec{z}^\gamma(i)$ and $\vec{z}^\gamma(i)\not=^*\vec{f}^\gamma(\eta_\gamma)(0)$;
\item $\vec{z}^\gamma(i)$ and $\vec{f}^\gamma(\eta_\gamma)(\tau)$ are mutually exclusive for every nonzero $\tau<\theta$.
\end{itemize}

This is possible because the combination of $\nu_{T^\gamma}\ge\max\{\mu,\theta^+\}$
with $\eta_\gamma=\eta_{\alpha+1}+1$ and $\supp(\vec f^{\alpha}(\eta_\alpha), \vec f^{\gamma}(\eta_\gamma))=\theta$ enables us to ensure
that $\langle \vec{z^\gamma}(i)(\eta_{\alpha+1})\mid \gamma<i<\mu\rangle{}^\smallfrown\langle \vec{f^\gamma}(\tau)(\eta_{\alpha+1})\mid \tau<\theta\rangle$
be an injective sequence, and $\langle \vec{z^\gamma}(i)\restriction\eta_{\alpha+1}\mid \gamma<i<\mu\rangle$
be equal to $\langle \vec z^\alpha(i)*(\vec{f^\gamma}(\eta_\gamma)(0)\restriction\eta_{\alpha+1})\mid \gamma<i<\mu\rangle$.

$\br$ If $\gamma$ is a limit ordinal, then obtain $\langle T^\gamma, \vec{f}^\gamma \rangle$ and $\vec z^\gamma$ by appealing to Lemma~\ref{minimalextension2}
with the sequence
$\langle (T^\beta,\vec{f^\beta}, \vec{z}^\beta) \mid \beta<\gamma\text{ nonzero even ordinal}\rangle$.
Then let $\eta_\gamma:=\eta_{T^\gamma}$.

In all cases, $( T^{\gamma},\vec{f}^{\gamma},\eta_\gamma)$ belongs to $D$ and requirements \ref{basic22}--\ref{basic12} are satisfied for $\gamma$
by the induction and the construction of $( T^{\gamma},\vec{f}^{\gamma},\eta_\gamma)$.
This completes the description of our winning strategy for $\pII$.
\end{proof}

\begin{lemma}\label{chainconditionlemma2}
For every $\xi<\zeta$,
$\mathbb S^\kappa_{\mathbf X} \Vdash \mathbb{\dot{A}}^\kappa_{\mathbf X,\xi}$ has the $\kappa$-cc.

In particular, $\mathbb S^\kappa_{\mathbf X}\Vdash \mathbb{\dot{A}}^\kappa_{\mathbf X}$ has the $\kappa$-cc.
\end{lemma}
\begin{proof} The ``in particular'' part follows from the fact that $\zeta<\theta^+<\kappa$.
Now, fix a $\xi<\zeta$,
and let $\dot{A}$ be an $\mathbb S^\kappa_{\mathbf X}$-name for a maximal antichain in $\mathbb{\dot{A}}^\kappa_{\mathbf X,\xi}$.
We shall need the following key claim concerning the poset $\mathbb S^\kappa_{\mathbf X}$.

\begin{claim}\label{Sealinclaim2} Let $\langle T,\vec{f}\rangle\in \mathbb S^\kappa_{\mathbf X}$.
Then there exists $\langle T',\vec{f'}\rangle\le\langle T,\vec{f}\rangle$
with $\supp(\vec{f}(\eta_T),\vec{f'}(\eta_{T'}))=\theta$ such that
for every triple $(\vec{x},Y,\pi)$ satisfying the following two
\begin{enumerate}
\item\label{2areq1} $\vec{x}=\langle x_\tau \mid \tau<\theta \rangle$ consists of mutually exclusive nodes of $T_{\eta_T}$;
\item\label{2areq2} $Y \in \mathcal I$ and $\pi:Y\rightarrow Y$ is an injection such that
\begin{itemize}
\item $x_\tau=\vec f(\eta_T)(\tau)$ for every $\tau\in\theta\setminus Y$;
\item $x_\tau=^*\vec f(\eta_T)(\pi(\tau))$ for every $\tau \in Y$.
\end{itemize}
\end{enumerate}
there exists an $\alpha<\eta_{T'}$ such that
\begin{itemize}
\item $\langle T',\vec{f'}\rangle\Vdash_{\mathbb S^\kappa_{\mathbf X}}\alpha\in \dot A$;
\item $X_\xi\setminus Y\s \supp(\vec{f'}(\alpha),\vec{f'}(\eta_{T'}))$;
\item $\vec{f'}(\alpha)(\tau) \s x_\tau*\vec{f'}(\eta_{T'})(\pi(\tau))$ for every $\tau \in X_\xi\cap Y$.
\end{itemize}
\end{claim}
\begin{proof} By our assumptions, $\mu:=\max\{|T_{\eta_{T}}|,\theta\}^\theta$ is smaller than $\kappa$.
Fix a bijection $h$ from the set of all even ordinals in $\mu$ to the set of all triples $(\vec{x},Y,\pi)$ satisfying \eqref{2areq1} and \eqref{2areq2}.
We shall construct a sequence $\langle (T^{j},\vec{f}^{j},\vec{z}^j)\mid j<\mu \rangle$
in such a way that all of the following hold:
\begin{enumerate}[label=(\roman*)]
\item\label{2ac1} $\langle T^{0},\vec{f^0}\rangle=\langle T,\vec f\rangle$;
\item for all $i<j<\mu$, $\langle T^{j},\vec{f}^{j}\rangle\le \langle T^i,\vec{f}^i\rangle$ and $X_\xi\s\supp(\vec f^{i}(\eta_{T^i}), \vec f^{j}(\eta_{T^j}))$;
\item\label{2aa1} for every pair $i<j$ of even ordinals below $\mu$, $\supp(\vec f^{i}(\eta_{T^i}), \vec f^{j}(\eta_{T^j}))=\theta$;
\item for every nonzero even $j<\mu$, $\vec{z}^j:(j,\mu]\rightarrow T^j$, where $j<\iota<\mu$ implies:
\begin{itemize}
\item $\vec{z}^j(\iota)$ belongs to the top level of $T^j$;
\item $\vec{z}^j(\iota)\not=^*\vec{z}^j(\iota')$ for every $\iota'\in(j,\mu]\setminus\{\iota\}$;
\item $\vec{z}^j(\iota)\not=^*\vec{f}^j(\eta_{T^j})(0)$;
\item $\vec{z}^j(\iota)$ and $\vec{f}^j(\eta_{T^j})(\tau)$ are mutually exclusive for every nonzero $\tau<\theta$;
\item $\vec z^i(\iota) \s \vec z^j(\iota)$ for every nonzero even $i<j$ and every $\iota\in(j,\mu]$.
\end{itemize}
\item\label{2ab1} for every even $i<\mu$,
letting $(\langle x_\tau\mid \tau<\theta\rangle,Y,\pi):=h(i)$, there exists an $\alpha\le\eta_{T^{i+1}}$ such that
\begin{enumerate}
\item $\langle T^{i+1},\vec{f}^{i+1}\rangle\Vdash_{\mathbb S^\kappa_{\mathbf X}}\alpha\in \dot A$;
\item\label{2subb} $X_\xi\setminus Y\s \supp(\vec{f}^{i+1}(\alpha),\vec{f}^{i+1}(\eta_{T^{i+1}}))$;
\item\label{2subc} $\vec{f}^{i+1}(\alpha)(\tau)\s x_\tau*\vec{f}^{i+1}(\eta_{T^{i+1}})(\pi(\tau))$
for every $\tau \in X_\xi\cap Y$.
\end{enumerate}
\end{enumerate}

We start by setting $(T^{0},\vec f^0,\vec z^0):=(T,\vec f,\emptyset)$.
Next, suppose that $j\in(0,\mu]$ is such that $\langle (T^{i},\vec{f}^{i},\vec{z}^i)\mid i<j \rangle$ has already been defined.

{\bf Case 1: $j=i+1$ for an even ordinal $i$:}
Let $(\langle x_\tau \mid \tau <\theta \rangle,Y,\pi):=h(i)$.
By \eqref{2areq2}, $\langle \vec f(\eta_T)(\tau)\mid \tau\in\theta\setminus Y\rangle{}^\smallfrown\langle x_\tau\mid \tau\in Y\rangle$ coincides with
$\langle x_\tau \mid \tau <\theta \rangle$. So, by \eqref{2areq1}, it consists of mutually exclusive nodes.
By Clauses \ref{2ac1} and \ref{2aa1}, $\supp(\vec{f}(\eta_{T}),\vec{f^i}(\eta_{T^i}))=\theta$.
So, since $Y\in\mathcal I$, we can take a one-step extension $\langle S,\vec g\rangle$ of $\langle T^{i}, \vec{f}^{i} \rangle$
such that
\begin{itemize}
\item $\theta\setminus Y\s \supp(\vec{f}^{i}(\eta_{T^i}),\vec g(\eta_S))$, and
\item for every $\tau \in Y$, $x_\tau*\vec{f}^{i}(\eta_{T^i})(\pi(\tau))\s \vec g(\eta_S)(\tau)$.
\end{itemize}

\begin{figure}
\begin{center}
\begin{tikzpicture}
\node (b) at(8,5.2){$\vec{f}^{j}(\eta_{T^j})(\pi(\tau))$};
\draw (8.2,3.6)--(8,4.8);
\draw(6,4.8)--(14,4.8);
\node (c) at(15,4.8){$\eta_{T^j}$};
\draw(6,4.2)--(14,4.2);
\node (d) at(15,4.2){$\eta_{S'}$};
\draw[red](6,3.6)--(14,3.6);
\node[red] (e) at(15,3.6){$\alpha \in A$};
\draw(6,2.5)--(14,2.5);
\node (f) at(15,2.5){$\eta_S$};
\draw(6,1.8)--(14,1.8);
\node (g) at(15,1.8){$\eta_{T^i}$};
\draw(6,1)--(14,1);
\node (h) at(15,1){$\eta_T$};
\node (i) at(8,0.7){$x_\tau$};
\node (j) at(10,0.7){$=^*$};
\node (k) at(12,0.7){$\vec{f}(\eta_T)(\pi(\tau))$};
\draw[dashed](8,1)--(7,1.8);
\draw (12,1)--(11,1.8);
\node (l) at(11.5,2.1){$\vec{f}^{i}(\eta_{T^i})(\pi(\tau))$};
\node[blue] (m) at(8.6,2.1){$x_\tau*\vec{f}^{i}(\eta_{T^i})(\pi(\tau))$};
\node (n) at(6.2,2.8){$\vec g(\eta_S)(\tau)$};
\draw (7,1.8)--(7,2.5);
\node (o) at(7.2,3.9){$\vec{g'}(\alpha)(\tau)$};
\draw (7,2.5)--(8.2,3.6);
\end{tikzpicture}
\caption{Case: $\tau \in X_\xi\cap Y$.}
\end{center}
\end{figure}
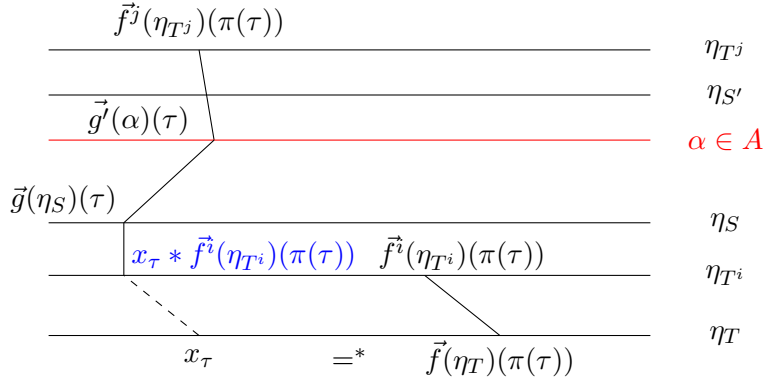
Recalling Definition~\ref{defA2},
and as $A$ is a maximal antichain, we may find a condition $\langle S', \vec{g'}\rangle\le\langle S,\vec g\rangle$
such that, for some $\alpha\le\eta_{S'}$,
\begin{itemize}
\item $\langle S',\vec{g'}\rangle\Vdash_{\mathbb S^\kappa_{\mathbf X}}\alpha\in \dot A$, and
\item $X_\xi\s \supp(\vec{g}(\eta_S),\vec{g'}(\alpha))$.
\end{itemize}

Fix an injection $\psi:\theta\setminus((X_\xi\setminus Y)\cup(X_\xi\cap\im(\pi)))\rightarrow\theta\setminus X_\xi$.\footnote{This is precisely where the hypothesis $|\theta\setminus X_0|=\theta$ comes into play.}
Denote $\alpha':=\max\{\eta_S,\alpha\}$.
Pick a one-step extension $\langle T^j,\vec{f^j}\rangle$ of $\langle S',\vec{g'}\rangle$
such that, for every $\tau<\theta$:
$$\vec{f}^{j}(\eta_{T^j})(\tau)=\begin{cases}
(\vec{g'}(\alpha')(\tau)*\vec{g'}(\eta_{S'})(\tau)){}^\smallfrown\langle \tau\rangle,&\text{if }\tau\in X_\xi\setminus Y;\\
(\vec{g'}(\alpha')(\pi^{-1}(\tau))*\vec{g'}(\eta_{S'})(\tau)){}^\smallfrown\langle\pi^{-1}(\tau)\rangle,&\text{if }\tau\in X_\xi\cap \im(\pi);\\
(\vec{g'}(\alpha')(\psi(\tau))*\vec{g'}(\eta_{S'})(\tau)){}^\smallfrown\langle\psi(\tau)\rangle,&\text{otherwise}.
\end{cases}$$

This ensures that the following two hold:
\begin{itemize}
\item $X_\xi\setminus Y\s \supp(\vec{g'}(\alpha),\vec{f}^{j}(\eta_{T^j}))$;
\item $\vec{g'}(\alpha)(\tau)\s x_\tau*\vec{f}^{j}(\eta_{T^j})(\pi(\tau))=\vec{f}^{j}(\eta_{T^j})(\pi(\tau))$ for every $\tau \in X_\xi\cap Y$.
\end{itemize}

{\bf Case 2: $j=2$:} Using Lemma~\ref{onestepextension2},
fix a one-step extension $\langle T^2,\vec{f}^2\rangle$ of $\langle T^1,\vec f^1\rangle$
such that $\nu_{T^2}\ge\mu$.
In particular, we may fix a map $\vec{z^2}:(2,\mu]\rightarrow T^2$, where $2<\iota<\mu$ implies:
\begin{itemize}
\item $\vec{z^2}(\iota)$ belongs to the top level of $T^2$;
\item $\vec{z^2}(\iota)\not=^*\vec{z^2}(\iota')$ for every $\iota'\in(2,\mu]\setminus\{\iota\}$;
\item $\vec{z^2}(\iota)\not=^*\vec{f}^2(\eta_{T^2})(0)$;
\item $\vec{z^2}(\iota)$ and $\vec{f}^2(\eta_{T^2})(\tau)$ are mutually exclusive for every nonzero $\tau<\theta$.
\end{itemize}

{\bf Case 3: $j=i+2$ for an even ordinal $i$:}
Using Lemma~\ref{onestepextension2},
let $\langle T^j,\vec{f^j}\rangle$ be a one-step $\eta_{T^i}$-extension of $\langle T^{i+1},\vec{f^{i+1}}\rangle$ with $\nu_{T^j}\ge\mu$.
As $\supp(\vec f^{i}(\eta_{T^i}), \vec f^{j}(\eta_{T^j}))=\theta$,
we may then fix a map $\vec{z}^j:(j,\mu]\rightarrow T^j$, where $j<\iota<\mu$ implies:
\begin{itemize}
\item $\vec{z}^j(\iota)$ belongs to the top level of $T^j$;
\item $\vec{z}^j(\iota)\not=^*\vec{z}^j(\iota')$ for every $\iota'\in(j,\mu]\setminus\{\iota\}$;
\item $\vec z^i(\iota)\s\vec{z}^j(\iota)$ and $\vec{z}^j(\iota)\not=^*\vec{f}^j(\eta_{T^j})(\tau)$;
\item $\vec{z}^j(\iota)$ and $\vec{f}^j(\eta_{T^j})(\tau)$ are mutually exclusive for every nonzero $\tau<\theta$.
\end{itemize}

{\bf Case 4: $j\in\acc(\mu+1)$:}
Obtain $(T^j,\vec{f^j},\vec z^j)$ by appealing to Lemma~\ref{minimalextension2}
with the sequence $\langle (T^i,\vec{f^i}, \vec{z}^i) \mid i<j\text{ nonzero even ordinal}\rangle$,
using $\delta:=\mu+1$.

\medskip

Once we are done with the recursion,
it is clear that $\langle T',\vec{f'}\rangle:=\langle T^\mu,\vec{f^\mu}\rangle$ is as sought.
\end{proof}

Let $\langle T,\vec{f}\rangle$ be an arbitrary condition in $\mathbb S^\kappa_{\mathbf X}$,
and we shall find an extension of it forcing that $A$ is bounded in $\kappa$.
To this end, we shall construct a sequence $\langle (T^{j},\vec{f}^{j},\vec{z}^j)\mid j<\theta^+ \rangle$
in such a way that all of the following hold:
\begin{enumerate}[label=(\roman*)]
\item $\langle T^{0},\vec{f^0}\rangle\le \langle T,\vec f\rangle$;
\item for all $i<j<\theta^+$, $(T^{j},\vec{f}^{j}\rangle\le (T^i,\vec{f}^i\rangle$ and $\supp(\vec f^{i}(\eta_{T^i}), \vec f^{j}(\eta_{T^j}))=\theta$;
\item for every $j<\theta^+$, $\vec{z}^j:(j,\theta^+]\rightarrow T^j$, where $j<\iota<\theta^+$ implies:
\begin{itemize}
\item $\vec{z}^j(\iota)$ belongs to the top level of $T^j$;
\item $\vec{z}^j(\iota)\not=^*\vec{z}^j(\iota')$ for every $\iota'\in(j,\theta^+]\setminus\{\iota\}$;
\item $\vec{z}^j(\iota)\not=^*\vec{f}^j(\eta_{T^j})(0)$;
\item $\vec{z}^j(\iota)$ and $\vec{f}^j(\eta_{T^j})(\tau)$ are mutually exclusive for every nonzero $\tau<\theta$;
\item $\vec z^i(\iota) \s \vec z^j(\iota)$ for every $i<j$ and every $\iota\in(j,\theta^+]$.
\end{itemize}
\end{enumerate}

We start by appealing to Lemma~\ref{onestepextension2} to find a condition $\langle T^0,\vec{f}^0\rangle\le \langle T,\vec f\rangle$ with $\nu_{T^0}\ge\theta^+$.
Then, fix a map $\vec{z}^0:(0,\theta^+]\rightarrow T^0$, where $0<\iota<\theta^+$ implies:
\begin{itemize}
\item $\vec{z}^0(\iota)$ belongs to the top level of $T^0$;
\item $\vec{z}^0(\iota)\not=^*\vec{z}^0(\iota')$ for every $\iota'\in(0,\theta^+]\setminus\{\iota\}$;
\item $\vec{z}^0(\iota)\not=^*\vec{f}^0(\eta_0)(0)$;
\item $\vec{z}^0(\iota)$ and $\vec{f}^0(\eta_0)(\tau)$ are mutually exclusive for every nonzero $\tau<\theta$.
\end{itemize}

Next, for every $i<\theta^+$ such that
$(T^i,\vec{f}^i,\vec{z}^i)$ has already been determined,
obtain $\langle T^{i+1},\vec f^{i+1}\rangle$ by appealing to Claim~\ref{Sealinclaim2}
with $\langle T^i,\vec{f}^i\rangle$.
By possibly replacing $\langle T^{i+1},\vec{f}^{i+1}\rangle$ with a one-step $\eta_{T^{i+1}}$-extension of it (using Lemma~\ref{onestepextension2}), we may also assume that $\eta_{T^{i+1}}$ is a successor ordinal and $\nu_{T^{i+1}}\ge\theta^+$.
Consequently, we may construct $\vec z^{i+1}$ satisfying requirement~(iii) for $j:=i+1$.
At limit stages, we obviously use Lemma~\ref{minimalextension2} with $\delta:=\theta^++1$.
Hereon, for notational simplicity, denote $\eta_{T^i}$ by $\eta_i$.

We claim that $\langle T^{\theta^+},\vec{f}^{\theta^+}\rangle$ forces that $A$ is bounded in $\kappa$, and in fact $A\s\eta_{\theta^+}$.
Suppose not, and pick a condition $\langle S,\vec{g}\rangle \le\langle T^{\theta^+},\vec{f}^{\theta^+}\rangle$ that
forces some $\beta\geq \eta_{\theta^+}$ to be in $A$. Without loss of generality, $\beta\le\eta_S$.

Consider $Y:=\theta\setminus\supp(\vec{g}(\beta),\vec{f}^{\theta^+}(\eta_{\theta^+}))$ which is an element of $\mathcal I$.
Recalling that we have invoked Lemma~\ref{minimalextension2} with $\delta:=\theta^++1$,
every element of $T_{\eta_{\theta^+}}$ is equal modulo bounded to $\vec{f}^{\theta^+}(\eta_{\theta^+})(\tau)$ for some $\tau<\theta$.
Consequently, we may fix a map $\pi:Y\rightarrow\theta$ such that
$\vec{g}(\beta)(\tau)\restriction \eta_{\theta^+}=^*\vec{f}^{\theta^+}(\eta_{\theta^+})(\pi(\tau))$ for every $\tau \in Y$.
As the elements of $\vec g(\beta)$ are mutually exclusive, it follows that $\pi$ is an injection from $Y$ to $Y$.
To summarize, we have found a $Y\in\mathcal I$ and an injection $\pi:Y\rightarrow Y$ such that the following two hold:
\begin{enumerate}[label=(\arabic*)]
\item\label{c1} $\vec{g}(\beta)(\tau)\restriction \eta_{\theta^+}=\vec{f}^{\theta^+}(\eta_{\theta^+})(\tau)$ for every $\tau\in\theta\setminus Y$;
\item $\vec{g}(\beta)(\tau)\restriction \eta_{\theta^+}=^*\vec{f}^{\theta^+}(\eta_{\theta^+})(\pi(\tau))$ for every $\tau \in Y$.
\setcounter{oldenumi}{\value{enumi}}
\end{enumerate}

Next, as $\dom(\vec{g}(\beta))=\theta$, we may find a large enough $i<\theta^+$ such that
\begin{enumerate}[label=(\arabic*)]
\setcounter{enumi}{\value{oldenumi}}
\item\label{c3} $\vec{g}(\beta)(\tau)\restriction \eta_{\theta^+}=(\vec{g}(\beta)(\tau)\restriction\eta_i)*\vec{f}^{\theta^+}(\eta_{\theta^+})(\pi(\tau))$ for every $\tau \in Y$.
\setcounter{oldenumi}{\value{enumi}}
\end{enumerate}
In particular:
\begin{enumerate}[label=(\arabic*)]
\setcounter{enumi}{\value{oldenumi}}
\item $\vec{g}(\beta)(\tau)\restriction \eta_{i+1}=\vec{g}(\eta_{i+1})(\tau)$ for every $\tau\in\theta\setminus Y$, and
\item $\vec{g}(\beta)(\tau)\restriction \eta_{i+1}=^*\vec{g}(\eta_{i+1})(\pi(\tau))$ for every $\tau \in Y$.
\setcounter{oldenumi}{\value{enumi}}
\end{enumerate}

Consider $\vec{x}:=\langle \vec{g}(\beta)(\tau) \restriction \eta_{i+1} \mid \tau <\theta \rangle$.
Then the triple $(\vec x,Y,\pi)$ satisfies requirements \eqref{2areq1} and \eqref{2areq2} with respect to $\langle T^{i+1},\vec{f^{i+1}}\rangle$.
Therefore,
there exists an $\alpha<\eta_{i+2}$ such that
\begin{enumerate}[label=(\arabic*)]
\setcounter{enumi}{\value{oldenumi}}
\item $\langle T^{i+2},\vec{f}^{i+2}\rangle\Vdash_{\mathbb S^\kappa_{\mathbf X}}\alpha\in \dot A$;
\item\label{c7} $X_\xi\setminus Y\s \supp(\vec{f}^{i+2}(\alpha),\vec{f}^{i+2}(\eta_{i+2}))$;
\item\label{c8} $\vec{f}^{i+2}(\alpha)(\tau) \s x_\tau*\vec{f}^{i+2}(\eta_{i+2})(\pi(\tau))$ for every $\tau \in X_\xi\cap Y$.
\setcounter{oldenumi}{\value{enumi}}
\end{enumerate}

By Clauses \ref{c7} and \ref{c1}, for every $\tau\in X_\xi\setminus Y$,
$$\vec{f}^{i+2}(\alpha)(\tau)\s \vec{f}^{i+2}(\eta_{i+2})(\tau)\s \vec{f}^{\theta^+}(\eta_{\theta^+})(\tau)\s \vec{g}(\beta)(\tau).$$

By Clauses \ref{c8} and \ref{c3}, for every $\tau \in X_\xi\cap Y$, we have
$$\begin{aligned}
\vec{f}^{i+2}(\alpha)(\tau)&\s x_\tau*\vec{f}^{i+2}(\eta_{i+2})(\pi(\tau))\\
&= (\vec{g}(\beta)(\tau)\restriction \eta_{i+1})*\vec{f}^{i+2}(\eta_{i+2})(\pi(\tau))\\
&\s(\vec{g}(\beta)(\tau)\restriction \eta_{i+1})*\vec{f}^{\theta^+}(\eta_{\theta^+})(\pi(\tau))\\
&=\vec{g}(\beta)(\tau)\restriction \eta_{\theta^+}\s \vec g(\beta)(\tau).
\end{aligned}$$

As $\vec{f^{i+2}}\s\vec{f}^{\theta^+}\s\vec g$,
we infer that $\langle S,\vec{g} \rangle$ forces that $\beta\in A$ is a proper extension of $\alpha\in A$, contradicting the fact
that $A$ is an antichain.
\end{proof}

\begin{cor}\label{lem6.5}
For every $\mathbb S^\kappa_{\mathbf X}$-generic filter $G$:
\begin{enumerate}[label=(\arabic*)]
\item $T(G)$ is a $\kappa$-Souslin tree;
\item $V(T(G))$ is a club in $\kappa$;
\item $\vec{f}^G$ is a mutually exclusive $\mathcal{F}$-ascent path through $T(G)$;
\item $T(G)$ coincides with the tree generated by its ascent path. Furthermore, for every
$t\in T(G)$ and every $\xi<\zeta$, there are $\alpha<\kappa$ and $\tau\in X_\xi$ such that $t\s \vec{f^G}(\alpha)(\tau)$.
\end{enumerate}
\end{cor}
\begin{proof} (1) $T(G)$ is a $\kappa$-tree thanks to Lemmas \ref{onestepextension2} and \ref{cor282}.
Towards a contradiction, suppose that $T(G)$ admits a $\kappa$-sized antichain.
It then follows from the upcoming Clause~(4) that there exists an injective sequence $\langle (\alpha_\gamma,\tau_\gamma)\mid \gamma<\kappa\rangle$ of elements of $\kappa\times X_0$
such that $\{ \vec{f^G}(\alpha_\gamma)(\tau_\gamma)\mid \gamma<\kappa\}$
is an antichain in $T(G)$. Find a $\tau\in X_0$ for which $\Gamma:=\{\gamma<\kappa\mid \tau_\gamma=\tau\}$ has size $\kappa$.
Then $\{\alpha_\gamma\mid\gamma\in\Gamma\}$ is a $\kappa$-sized antichain in $\mathbb A^\kappa_{\mathbf X,0}$,
contradicting Lemma~\ref{chainconditionlemma2}.

(2) By Definition~\ref{def31}\ref{vanishingclause},
it suffices to prove that for every $\varepsilon<\kappa$, there exists a $\langle T,\vec f\rangle\in G$ such that $\eta_T\in\acc(\kappa\setminus\varepsilon)$.
Now, this follows from Lemma~\ref{minimalextension2}, as demonstrated by the proof of Lemma~\ref{Strclosure2}.

(3) By Definition~\ref{def31}\ref{ascentpathclause}.

(4) By a density argument, using Definition~\ref{def31}\ref{keyclause}.
\end{proof}

\section{On the verge of weak compactness}\label{sect7}

For the purpose of this section, we fix an $\mathbf X=\langle X_n\mid n<\omega\rangle$ satisfying the following requirements:
\begin{itemize}
\item $\mathbf X$ is a $\s$-decreasing sequence of nonempty subsets of $\omega$;
\item $|\omega\setminus X_0|=|X_0\setminus X_1|=\omega$;
\item $\bigcap_{n<\omega}X_n=\emptyset$.
\end{itemize}
As in the previous section, we denote by $\mathcal F$ the filter generated by $\mathbf X$,
and we derive the corresponding notions of forcing of Definitions \ref{def31} and \ref{defA2}.
The next result is the main theorem of the paper.

\begin{thm}\label{thm77}
Suppose that $\kappa$ is a weakly compact cardinal. In some forcing extension, $\kappa$ is a strongly inaccessible cardinal
satisfying $\chi(\kappa)=1$,
and there exists a $\kappa$-Souslin tree $T$ such that $V(T)=\acc(\kappa)$.
\end{thm}
\begin{proof} For an ordinal $\tau$,
let $\mathbb P_\tau$ be the Easton-support iteration of length $\tau$ which forces
with $\mathbb S^\eta_{\mathbf X}*\mathbb A^\eta_{\mathbf X}$ at every strongly inaccessible cardinal $\eta<\tau$.
As usual, at $\eta<\tau$ that is not strongly inaccessible, we use trivial forcing.

Let $G_\kappa$ be $\mathbb P_\kappa$-generic over $V$.
Let $g$ be $\mathbb S^\kappa_{\mathbf X}$-generic over $V[G_\kappa]$
and let $h$ be $\mathbb A^\kappa_{\mathbf X, 0}$-generic over $V[G_\kappa][g]$.
Note that $G_{\kappa+1}:=G_\kappa*g*h$ is $\mathbb P_{\kappa+1}$-generic over $V$.

By Corollary~\ref{lem6.5}, in $V[G_\kappa][g]$, $T(g)$ is a uniformly homogeneous $\kappa$-Souslin tree
admitting an $\mathcal F$-ascent path $\vec f^g$, and $V(T(g))$ covers a club in $\kappa$.
By \cite[Lemma~2.5]{paper58}, then, $$V[G_\kappa][g]\models\text{``there exists a }\kappa\text{-Souslin tree }T\text{ such that }V(T)=\acc(\kappa)".$$
It thus remains to prove that $\chi(\kappa)\le 1$ holds in $V[G_\kappa][g]$.
To this end, let $\vec{C}=\langle C_\beta \mid \beta<\kappa \rangle$ be an arbitrary $C$-sequence in $V[G_\kappa][g]$.

Work in $V$. Let $\dot{C}$ be a $\mathbb P_\kappa*\dot{\mathbb S}^\kappa_{\mathbf X}$-name for $\vec{C}$.
Take a $\kappa$-model $M$ containing $\mathbb P_\kappa*\dot{\mathbb S}^\kappa_{\mathbf X}*\dot{\mathbb A}^{\kappa}_{\mathbf X}$ and $\dot{C}$.
As $\kappa$ is weakly compact, we may now pick a $\kappa$-model $N$
and a nontrivial elementary embedding $j_0:M \rightarrow N$ with critical point $\kappa$.
Hereafter, work in $V[G_{\kappa+1}]$.
\begin{claim} $j_0$ may be lifted to an elementary embedding
$$j_1:M[G_\kappa]\rightarrow N[G_{\kappa+1}*G_{tail}].$$
\end{claim}
\begin{proof} By \cite[Proposition~7.13 and Remark~7.14]{cummings},
$\mathbb P_\kappa$ has the $\kappa$-cc.
In addition, by Lemma~\ref{Strclosure2}, in $V[G_\kappa]$, $\mathbb S^\kappa_{\mathbf X}*\mathbb A^\kappa_{\mathbf X}$ has a dense set that is ${<}\kappa$-strategically closed.
So, from ${}^{<\kappa}N \s N$, we infer that ${}^{<\kappa}N[G_{\kappa+1}] \s N[G_{\kappa+1}]$.
Since we use nontrivial forcing only at strongly inaccessibles, Lemma~\ref{Strclosure2} implies that
the quotient forcing $j_0(\mathbb P_\kappa)/ G_{\kappa+1}$ is $\kappa$-strategically closed in $N[G_{\kappa+1}]$.
But ${}^{<\kappa}N[G_{\kappa+1}] \s N[G_{\kappa+1}]$, and hence $j_0(\mathbb P_\kappa)/ G_{\kappa+1}$ is $\kappa$-strategically closed.
As $N$ has no more than $\kappa$ many $\mathbb P_{\kappa+1}$-names, the model $N[G_{\kappa+1}]$ has no more than $\kappa$ many dense sets of $j_0(\mathbb P_\kappa)/ G_{\kappa+1}$.
It thus follows that a $j_0(\mathbb P_\kappa)/ G_{\kappa+1}$-generic over $N[G_{\kappa+1}]$ may be recursively constructed,
say $G_{tail}$, so that $N[G_{\kappa+1}][G_{tail}]$ is a $j_0(\mathbb P_\kappa)$-generic extension of $N$.
As $j_0$ is the identity map over $\mathbb P_\kappa$, Silver's criterion holds vacuously,
so we may lift $j_0$ to a $j_1:M[G_\kappa]\rightarrow N[G_{\kappa+1}*G_{tail}]$, as sought.
\end{proof}

\begin{claim} $j_1$ may be lifted to an elementary embedding
$$j_2:M[G_\kappa*g]\rightarrow N[G_{\kappa+1}*G_{tail}*g_{tail}].$$
\end{claim}
\begin{proof} For every $n\in X_0$, define $b_n:\kappa\rightarrow\kappa$ via:
$$b_n:=\bigcup\{\vec f^g(\alpha)(n)\mid \alpha\in h\}.$$
As $h$ is $\mathbb A^\kappa_{\mathbf X, 0}$-generic over $V[G_\kappa][g]$,
each such $b_n$ is a cofinal branch through $T(g)$.
Furthermore, the elements of $\langle b_n\mid n\in X_0\rangle$ are mutually exclusive.
Take $n_0 \in X_0\setminus X_1$ and then let
$$T:=T(g)\cup\{ x*b_n\mid x\in T(g), n\in X_0\ \&\ n \neq n_0\}.$$
Then $T$ is a uniformly homogeneous tree for which $b_{n_0}$ is a vanishing $\kappa$-branch.
Thus, $V(T)=V(T(g))\cup\{\kappa\}$.
Next, we extend the ascent path $\vec{f}^g$ by setting $\vec{f}:=\vec{f}^g\cup\{ (\kappa,\langle b_{\pi(n)}\mid n<\omega\rangle)\}$,
where $\pi:\omega \rightarrow X_0 \setminus \{n_0\}$ is some bijection such that $\pi \restriction X_1$ is the identity function.

Altogether, $\langle T, \vec{f} \rangle$ is a legitimate condition of $\mathbb S^{j_1(\kappa)}_{\mathbf X}$
as computed in $N[G_{\kappa+1}]$,
in particular, it is an element of the last iterand of $j_1(\mathbb P_\kappa *\dot{\mathbb S}^\kappa_{\mathbf X})/ G_{\kappa+1}*G_{tail}$.
In addition, $\langle T, \vec{f} \rangle$ extends every condition lying in $j_1[g]$.

Since ${}^{<\kappa}N[G_{\kappa+1}] \s N[G_{\kappa+1}]$ and $G_{tail}$ is a generic for a $\kappa$-strategically closed forcing,
${}^{<\kappa}N[G_{\kappa+1}*G_{tail}] \s N[G_{\kappa+1}*G_{tail}]$.
By Lemma~\ref{Strclosure2}, $\mathbb S^{j_1(\kappa)}_{\mathbf X}$ is ${<}j_1(\kappa)$-strategically closed in $N[G_{\kappa+1}*G_{tail}]$.
But ${}^{<\kappa}N[G_{\kappa+1}*G_{tail}] \s N[G_{\kappa+1}*G_{tail}]$, and hence $\mathbb S^{j_1(\kappa)}_{\mathbf X}$ is $\kappa$-strategically closed.
As $N$ has no more than $\kappa$ many $j_1(\mathbb P_{\kappa})$-names, the model $N[G_{\kappa+1}*G_{tail}]$ has no more than $\kappa$ many dense sets of
$j_1(\mathbb P_\kappa *\dot{\mathbb S}^\kappa_{\mathbf X})/ G_{\kappa+1}*G_{tail}$.
It thus follows that a $j_1(\mathbb P_\kappa *\dot{\mathbb S}^\kappa_{\mathbf X})/ G_{\kappa+1}*G_{tail}$-generic over $N[G_{\kappa+1}*G_{tail}]$ may be recursively constructed,
say $g_{tail}$, so that $N[G_{\kappa+1}*G_{tail}][g_{tail}]$ is a $j_1(\mathbb P_\kappa*\dot{\mathbb S}^\kappa_{\mathbf X})$-generic extension of $N$
and $\langle T,\vec f\rangle\in g_{tail}$.
By Silver's criterion, we may now lift $j_1$ to a $j_2:M[G_\kappa][g]\rightarrow N[G_{\kappa+1}*G_{tail}][g_{tail}]$, as sought.
\end{proof}

Let $j_2:M[G_\kappa][g]\rightarrow N[G_{\kappa+1}*G_{tail}*g_{tail}]$ be given by the preceding claim.
Clearly, $\vec C$ is in $M[G_\kappa][g]$, so that $j_2(\vec C)$ is a $C$-sequence over $j_2(\kappa)$.
Let $C_\kappa$ denote $j_2(\vec{C})(\kappa)$.

\begin{claim} For every $\alpha<\kappa$, there exists $\beta\ge\alpha$ such that $C_\kappa\cap\alpha=C_\beta\cap\alpha$.
\end{claim}
\begin{proof} Suppose not, and let $\alpha<\kappa$ be a counterexample.
Denote $A:=C_\kappa\cap\alpha$, and define a map $\varphi:\kappa\setminus\alpha\rightarrow\alpha$ via:
$$\varphi(\beta):=\min\{\varepsilon<\alpha\mid A\cap\varepsilon\neq C_\beta\cap\varepsilon\}.$$

By elementarity, $j_2(A)\cap j_2(\varphi)(\kappa)\neq C_\kappa\cap j_2(\varphi)(\kappa)$.
However, $j_2(A)=A=C_\kappa\cap\alpha\s\alpha>j_2(\varphi)(\kappa)$. This is a contradiction.
\end{proof}

Recall that $C_\kappa$ is a club in $\kappa$ lying in $V[G_\kappa][g][h]$.
As $h$ is $\mathbb A^\kappa_{\mathbf X}$-generic over $V[G_\kappa][g]$, and as $\mathbb A^\kappa_{\mathbf X}$ satisfies the $\kappa$-cc,
it follows that there exists a club $D \s C_\kappa$ with $D\in V[G][g]$.
Evidently, for every $\alpha<\kappa$, there exists a $\beta\ge\alpha$ such that $D \cap \alpha \s C_\beta \cap \alpha$.
\end{proof}
\begin{remark}
The proof of Theorem~\ref{thm77} may easily be adapted to show that if $\kappa$ is supercompact and $\theta$ is measurable cardinal below it,
then there is a forcing extension in which $\kappa$ is a $\theta$-strongly compact cardinal
and $\chi(\kappa)=1$. A second interesting aspect of the proof of Theorem~\ref{thm77} is that the final model is an intersection model in the sense of \cite{paper69}.
In particular, an adaptation of the proof of \cite[Corollary~4.7]{paper69} to generic elementary embeddings yields that in the model of Theorem~\ref{thm77}, every $\kappa$-Aronszajn tree contains an $\omega$-ascent path.
Note that this is the first example of a model obtained as the decreasing intersection of countably many models in which $\chi(\kappa)<\omega$.
\end{remark}

\begin{cor} Assuming the consistency of a weakly compact cardinal, it is consistent that for some inaccessible cardinal $\kappa$ (1) holds, but (2) fails:
\begin{enumerate}[label=(\arabic*)]
\item For every $C$-sequence $\langle C_\beta\mid\beta<\kappa\rangle$,
there exists a cofinal $\Delta\s\kappa$ such that, for every $\alpha<\kappa$, there exists a $\beta<\kappa$ with $\Delta\cap\alpha\s C_\beta$;
\item For every $C$-sequence $\langle C_\beta\mid\beta<\kappa\rangle$,
there exists a cofinal $\Delta\s\kappa$ such that, for every $\alpha<\kappa$, there exists a $\beta<\kappa$ with $\Delta\cap\alpha\s\nacc(C_\beta)$.\footnote{The shift of attention to $\nacc(C_\beta)$
amounts to replacing the quantification from ranging over $C$-sequences to ranging over ladder systems.}
\end{enumerate}
\end{cor}
\begin{proof} Work in the model of Theorem~\ref{thm77}. We have $\chi(\kappa)=1$ which amounts to saying that Clause~(1) holds.

In addition, there exists a $\kappa$-Souslin tree $T$ such that $V(T)=\acc(\kappa)$.
An inspection of the proof of \cite[Theorem~2.23]{paper48} makes it clear that there exists a club $D\s\kappa$
such that for every partition $\mathcal S$ of $D$ into stationary sets, $\clubsuit_{\ad}(\mathcal S,\omega,{<}\omega)$ holds.
In particular, $\clubsuit_{\ad}(\{D\},1,1)$ holds.
An inspection of the proof of \cite[Theorem~2.30]{paper48} then yields that Clause~(2) must fail.
\end{proof}

\section*{Acknowledgments}

The first and second authors were supported by the Israel Science Foundation (grant agreement 203/22).
The results of this paper were presented by the second author in an invited talk at the Chinese Annual Conference on Mathematical Logic, July 2023,
and in a seminar talk at the Institute of Mathematics, Chinese Academy of Sciences, November 2023.

\end{document}